\documentclass[10pt]{article}
\usepackage{amsmath, amsthm}\usepackage{enumerate}
\usepackage{amssymb}\usepackage{bbold}
\usepackage{color}
\newtheorem{theorem}{Theorem}[section]
\newtheorem{proposition}[theorem]{Proposition}

\newtheorem{lemma}[theorem]{Lemma}

\newtheorem{corollary}[theorem]{Corollary}

\newtheorem{definition}[theorem]{Definition}

\DeclareMathOperator{\convo}{\xrightarrow[]{o}}

\DeclareMathOperator{\convr}{\xrightarrow[]{ru}}

\DeclareMathOperator{\convw}{\xrightarrow[]{w}}

\DeclareMathOperator{\convn}{\xrightarrow[]{\|\cdot\|}}

\large
\makeatletter
\renewcommand{\subsection}{\@startsection{subsection}{1}
{0pt}{3.25ex plus 1ex minus.2ex}{-1em}{\normalfont\normalsize\bf}}
\makeatother
\begin{document}

\title{{\bf On automatic boundedness of some operators in ordered Banach spaces}}
\maketitle
\author{\centering{{Eduard Emelyanov$^{1}$\\ 
\small $1$ Sobolev Institute of Mathematics}
\abstract{We study order-to-weak continuous operators from an ordered Banach space to a normed space. 
It is proved that under rather mild conditions every order-to-weak continuous operator is bounded.}

\vspace{3mm}

\vspace{3mm}
}}

\section{Introduction and preliminaries}

\hspace{4mm}
Order-to-topology continuous and $\tau$-Lebesgue operators between vector lattices and topological vector spaces 
were studied recently in \cite{AEG2022,GE2022,JAM2021,KTA2025,ZSC2023}.
We adopt the definition to the setting of ordered vector spaces in the domain and normed spaces in the range.
It is important to find conditions on operators in ordered Banach spaces which provide automatic boundedness 
(see, for example \cite{AB2006,AT2007,AN2009} and the references therein). 
In the present paper, we continue the search of such conditions following \cite{E0-2025,EEG2025}.

We abbreviate ordered vector (normed, ordered vector, ordered Banach) spaces by \text{\rm OVSs} (\text{\rm NSs}, \text{\rm ONSs}, \text{\rm OBSs}) and 
vector (normed, Banach) lattices by \text{\rm VLs} (\text{\rm NLs}, \text{\rm BLs}).
Throughout the present paper, vector spaces are real, operators are linear,
symbol ${\cal L}(X,Y)$ stands for the space of operators from a vector space $X$ to a vector space $Y$, 
${\cal L}_{ob}(X,Y)$ for the space of order bounded operators from an OVS $X$ to an OVS $Y$,
$\text{\rm L}(X,Y)$ for the space of bounded operators from a NS $X$ to a NS $Y$, 
$x_\alpha\downarrow 0$ for a decreasing net in an OVS such that $\inf\limits_\alpha x_\alpha=0$, and $B_X$ for the 
closed unit ball of a NS $X$. A net $(x_\alpha)$ in an OVS $X$
\begin{enumerate}[-]
\item\ 
{\em order converges} to $x\in X$ 
(o-converges to $x$, or $x_\alpha\convo x$) whenever there exists 
$g_\beta\downarrow 0$ in $X$
such that, for each $\beta$ there is  $\alpha_\beta$ 
satisfying $\pm(x_\alpha-x)\le g_\beta$ for all $\alpha\ge\alpha_\beta$. 
\item\ 
{\em relative uniform converges} to $x\in X$ 
(ru-converges to $x$, or $x_\alpha\convr x$) 
if, for some $u\in X_+$ there exists an increasing sequence 
$(\alpha_n)$ of indices with $\pm(x_\alpha-x)\le\frac{1}{n}u$ for
all $\alpha\ge\alpha_n$. 
\end{enumerate}

\noindent
An operator $T$ from an OVS $X$ to an OVS $Y$ is {\em order continuous} if $Tx_\alpha\convo 0$ in $Y$
whenever $x_\alpha\convo 0$ in $X$ (the set of such operators is a vector space which is denoted by ${\cal L}_{oc}(X,Y)$). 

\begin{definition}\label{order-to-topology}
An operator $T$ from an OVS $X$ to a NS $Y$ is 
\begin{enumerate}[$a)$]
\item\
$(\sigma$-$)${\em Lebesgue} if $Tx_\alpha\convn 0$ for every $($sequence$)$ net $x_\alpha\downarrow 0$ in $X$ {\em \cite[Definition 1.1]{AEG2022}}.
The set of such operators is denoted by $\text{\rm L}_{Leb}(X,Y)$ $(\text{\rm L}^\sigma_{Leb}(X,Y))$. 
\item\
$(\sigma$-$)$\text{\rm w}-{\em Lebesgue} if $Tx_\alpha\convw 0$ for every $($sequence$)$ net $x_\alpha\downarrow 0$ in $X$ {\em \cite[Definition 1.1]{AEG2022}}.
The set of such operators is denoted by $\text{\rm L}_{wLeb}(X,Y)$ $(\text{\rm L}^\sigma_{wLeb}(X,Y))$. 
\item\
$(\sigma$-$)${\em order-to-norm continuous} if $Tx_\alpha\convn 0$ for every $($sequence$)$ net $x_\alpha\convo 0$ in $X$ {\em \cite{JAM2021}}.
The set of such operators is denoted by $\text{\rm L}_{onc}(X,Y)$ $(\text{\rm L}^\sigma_{onc}(X,Y))$. 
\item\
$(\sigma$-$)${\em order-to-weak continuous} if $Tx_\alpha\convw 0$ for every $($sequence$)$ net $x_\alpha\convo 0$ in $X$.
The set of such operators is denoted by $\text{\rm L}_{owc}(X,Y)$ $(\text{\rm L}^\sigma_{owc}(X,Y))$ {\em \cite{JAM2021}}. 
\item\
{\em order-to-norm bounded} if $T[a,b]$ is norm bounded for every order interval $[a,b]$ in $X$ {\em \cite[Definition 1.1]{EEG2025}}.
The set of such operators is denoted by $\text{\rm L}_{onb}(X,Y)$. 
\end{enumerate}
\end{definition}

\noindent
A positive operator $T$ from a VL to a NL is Lebesgue/$\sigma$-Lebesgue iff $T$ is 
order-to-norm continuous/$\sigma$-order-to-norm continuous by \cite[Lemma 2.1]{AEG2022} (see, also \cite[Proposition 1.5]{KTA2025}). 
The same fact holds under slightly more general assumptions.

\begin{lemma}\label{positive 1}
Let $X$ be a OVS, $Y$ a normal ONS, and $T\in{\cal L}_+(X,Y)$. Then 
$$
   T\in\text{\rm L}_{Leb}(X,Y)\Longleftrightarrow T\in\text{\rm L}_{onc}(X,Y) \ \ \ and \ \ \
   T\in\text{\rm L}^\sigma_{Leb}(X,Y)\Longleftrightarrow T\in\text{\rm L}^\sigma_{onc}(X,Y).
$$
\end{lemma}

\begin{proof}
As the $\sigma$-Lebesgue case is similar, we consider only the first equivalence. 
It suffices to show $\text{\rm L}_{Leb}(X,Y)\subseteq\text{\rm L}_{onc}(X,Y)$.
Suppose $T\in\text{\rm L}_{Leb}(X,Y)$ and take $x_\alpha\convo 0$ in $X$.
Pick a net $g_\beta\downarrow 0$ in $X$ such that, for each $\beta$, there exists $\alpha_\beta$
with $\pm x_\alpha\le g_\beta$ for $\alpha\ge\alpha_\beta$. Since $T\ge 0$,
we obtain $\pm Tx_\alpha\le Tg_\beta$ for $\alpha\ge\alpha_\beta$. 
As $T$ is Lebesgue, $\|Tg_\beta\|\to 0$. So, since $Tg_\beta\downarrow\ge 0$ then $Tg_\beta\downarrow 0$.
It follows from \cite[Theorem 2.23]{AT2007} that $\|Tx_\alpha\|\to 0$. Thus, $T\in\text{\rm L}_{onc}(X,Y)$.
\end{proof}

\noindent
By \cite[Proposition 2.2]{AEG2022}, a positive order continuous operator $T$ from a VL to a NL 
is Lebesgue iff $T$ is \text{\rm w}-Lebesgue. The next lemma extends this fact to the setting of OVSs.

\begin{lemma}\label{positive 2}
Let $X$ be a OVS, $Y$ a normal ONS, and $T\in({\cal L}_{oc})_+(X,Y)$. Then 
$$
   T\in\text{\rm L}_{Leb}(X,Y)\Longleftrightarrow T\in\text{\rm L}_{wLeb}(X,Y) \ \ \ and \ \ \
   T\in\text{\rm L}^\sigma_{Leb}(X,Y)\Longleftrightarrow T\in\text{\rm L}^\sigma_{wLeb}(X,Y).
$$
\end{lemma}

\begin{proof}
We consider only the Lebesgue case. It suffices to show $\text{\rm L}_{wLeb}(X,Y)\subseteq\text{\rm L}_{Leb}(X,Y)$.
Suppose $T\in\text{\rm L}_{wLeb}(X,Y)$ and take $x_\alpha\downarrow 0$ in $X$.
Since $T\ge 0$ then $Tx_\alpha\downarrow\ge 0$, and as $T$ is \text{\rm w}-Lebesgue, $Tx_\alpha\convw 0$. 
By \cite[Lemma 2.28]{AT2007}, we conclude $\|Tx_\alpha\|\to 0$. Thus, $T\in\text{\rm L}_{Leb}(X,Y)$.
\end{proof}

\noindent
The following lemma is a direct consequence of Definition \ref{order-to-topology}.

\begin{lemma}\label{directly}
Let $X$ be an OVS and $Y$ a NS. Then 
$$
   \text{\rm L}_{Leb}(X,Y)\subseteq\text{\rm L}_{wLeb}(X,Y), \ \ 
   \text{\rm L}^\sigma_{Leb}(X,Y)\subseteq\text{\rm L}^\sigma_{wLeb}(X,Y), 
$$
$$
   \text{\rm L}_{onc}(X,Y)\subseteq\text{\rm L}_{owc}(X,Y), \ \ 
   \text{\rm L}^\sigma_{onc}(X,Y)\subseteq\text{\rm L}^\sigma_{owc}(X,Y),  
$$
$$
   \text{\rm L}_{onc}(X,Y)\subseteq\text{\rm L}_{Leb}(X,Y)\bigcap\text{\rm L}^\sigma_{onc}(X,Y)\subseteq 
   \text{\rm L}_{Leb}(X,Y)\bigcup\text{\rm L}^\sigma_{onc}(X,Y)\subseteq\text{\rm L}^\sigma_{Leb}(X,Y), \ \ \text{\rm and}
$$
$$
   \text{\rm L}_{owc}(X,Y)\subseteq\text{\rm L}_{wLeb}(X,Y)\bigcap\text{\rm L}^\sigma_{owc}(X,Y)\subseteq 
   \text{\rm L}_{wLeb}(X,Y)\bigcup\text{\rm L}^\sigma_{owc}(X,Y)\subseteq\text{\rm L}^\sigma_{wLeb}(X,Y).
$$
\end{lemma}

\begin{proposition}\label{onb = bdd}
Let $X$ be an OBS with a closed generating normal cone and $Y$ a NS. Then $\text{\rm L}_{onb}(X,Y)=\text{\rm L}(X,Y)$.
\end{proposition}

\begin{proof}
The inclusion $\text{\rm L}_{onb}(X,Y)\subseteq\text{\rm L}(X,Y)$ follows from \cite[Corollary 2.9]{EEG2025}. 
The reverse inclusion hods due to normality of $X$.
\end{proof}

\medskip
The present note is organized as follows.
Lemma \ref{o-cont are obdd} states that every $\sigma$-\text{\rm w}-Lebesgue operator is order-to-norm bounded whenever OBS in the domain has normal cone.
Theorem \ref{OBS o-cont are obdd} states that if the domain has closed generating normal cone then every $\sigma$-\text{\rm w}-Lebesgue operator is bounded.
Theorem \ref{Lebesgue are onc} gives conditions under which ($\sigma$-)(\text{\rm w}-)Lebesgue operators are ($\sigma$-)order to (weak) norm continuous.


For the terminology and notations that are not explained in the text, we refer to \cite{AB2006,AT2007}.

\section{\large Main results}

\hspace{4mm}
It is established in \cite[Theorem 2.8]{EEG2025} that every order-to-norm bounded operator from 
an OBS with closed generating cone to a NS is bounded (see also \cite[Theorem 2.1]{E0-2025} for the Banach lattice case). 
Accordingly to \cite[Theorem 2.4]{EEG2025}, each order bounded operator from an OVS to an OVS with generating cones is ru-continuous.
Combining these two results and keeping in mind that order intervals in an OBS with normal cone are bounded, we obtain the following automatic 
boundedness conditions of order bounded operators: 

\begin{enumerate}[(*)]
\item\
${\cal L}_{ob}(X,Y)\subseteq\text{\rm L}(X,Y)$ whenever $X$ and $Y$ are OBSs with generating cones such that $X_+$ is closed and $Y_+$ is normal.
\end{enumerate}

\noindent
Another case of automatic boundedness of operators concerns the order continuity. 
Namely, as every order continuous operator from an Archimedean OVS with generating cone to an OVS is order bounded \cite[Corollary 2.2]{EEG2025}, we have:

\begin{enumerate}[(**)]
\item\
${\cal L}_{oc}(X,Y)\subseteq\text{\rm L}(X,Y)$ whenever $X$ and $Y$ are OBSs with generating cones such that $X_+$ is closed and $Y_+$ is normal.
\end{enumerate}

\noindent 
Condition $X_+-X_+=X$ is essential in both (*) and (**) since each operator $T:X\to X$ is order bounded and order continuous whenever $X_+=\{0\}$.
An inspection of \cite[Example 2.12 b)]{EEG2025} tells us that each infinite dimensional Banach space $X$ has a non-closed generating Archimedean cone so that
every operator $T:X\to X$ is order continuous (and hence order bounded by \cite[Corollary 2.2]{EEG2025}).
 
\medskip
\noindent 
It is worth nothing that the order-to-norm and order-to-weak continuous operators on VLs were defined in \cite{JAM2021}
without specifying the class of (just linear or continuous linear) operators.
Automatic boundedness of such operators was investigated in recent paper \cite{KTA2025}. 
In this section we establish automatic boundedness conditions for order-to-weak continuous operators on OVSs. More precisely,
Theorem \ref{OBS o-cont are obdd} which tells us that every $\sigma$-\text{\rm w}-Lebesgue operator from an OBS with 
a closed generating normal cone to a NS is bounded.

\medskip
We begin with conditions under which $\sigma$-\text{\rm w}-Lebesgue operators are order-to-norm bounded. 

\begin{lemma}\label{o-cont are obdd}
Let $X$ be a normal OBS and $Y$ a NS. Then $\text{\rm L}^\sigma_{wLeb}(X,Y)\subseteq\text{\rm L}_{onb}(X,Y)$. 
In particular, $\text{\rm L}^\sigma_{owc}(X,Y)\subseteq\text{\rm L}_{onb}(X,Y)$.
\end{lemma}

\begin{proof}
Let $T:X\to Y$ be $\sigma$-\text{\rm w}-Lebesgue. If $T$ is not order-to-norm bounded then $T[0,u]$ is not bounded for some $u\in X_+$. 
Since $X$ is normal then $[0,u]$ is bounded, say $\sup\limits_{x\in[0,u]}\|x\|\le M$. 
Pick a sequence $(u_n)$ in $[0,u]$ with $\|Tu_n\|\ge n2^n$, and set $y_n:=\|\cdot\|$-$\sum\limits_{k=n}^\infty 2^{-k}u_k$ for $n\in\mathbb{N}$.
Then $y_n\downarrow\ge 0$. Let $0\le y_0\le y_n$ for all $n\in\mathbb{N}$.
Since $0\le y_0\le 2^{1-n}u$ and $\|2^{1-n}u\|\le M2^{1-n}\to 0$ then $y_0=0$ by \cite[Theorem 2.23]{AT2007}.
Thus, $y_n\downarrow 0$. It follows from $T\in\text{\rm L}^\sigma_{wLeb}(X,Y)$ that $(Ty_n)$ is \text{\rm w}-null, 
and hence is norm bounded, that is absurd because
$\|Ty_{n+1}-Ty_n\|=\|T(2^{-n}u_n)\|\ge n$ for all $n\in\mathbb{N}$. 

The second inclusion follows from Lemma \ref{directly}.
\end{proof}

\noindent
Combining Lemma \ref{o-cont are obdd} with Proposition \ref{onb = bdd} and Lemma \ref{directly}, we obtain.

\begin{theorem}\label{OBS o-cont are obdd}
Let $X$ be an OBS with a closed generating normal cone and $Y$ a NS. 
Then $\text{\rm L}^\sigma_{wLeb}(E,Y)\subseteq\text{\rm L}(E,Y)$.
In particular, $\text{\rm L}^\sigma_{ows}(E,Y)\subseteq\text{\rm L}(E,Y)$.
\end{theorem}

\begin{corollary}\label{BL onc are cont}
{\rm (cf., \cite[Proposition 3.2]{AEG2022},\cite[Lemma 2.3]{KTA2025})}
Each $\sigma$-\text{\rm w}-Lebesgue $($and hence each $\sigma$-order-to-weak continuous$)$ operator from a BL to a NS is bounded.
\end{corollary}

\noindent
The next proposition is a straightforward consequence of Lemma \ref{directly} and Theorem \ref{OBS o-cont are obdd}.

\begin{proposition}\label{norm closed subspaces}
Let $X$ be an OBS with a closed generating normal cone and $Y$ a NS. Then
$\text{\rm L}_{Leb}(X,Y)$, $\text{\rm L}^\sigma_{Leb}(X,Y)$, $\text{\rm L}_{wLeb}(X,Y)$, $\text{\rm L}^\sigma_{wLeb}(X,Y)$,
$\text{\rm L}_{onc}(X,Y)$, $\text{\rm L}^\sigma_{onc}(X,Y)$, $\text{\rm L}_{owc}(X,Y)$, and $\text{\rm L}^\sigma_{owc}(X,Y)$
are norm closed subspaces of $\text{\rm L}(X,Y)$.
\end{proposition}

\noindent
Positive order-to-norm continuous order-to-norm bounded operators between NLs are not necessary bounded 
(see \cite[Remark 1.1.c)]{AEG2022}). From the other hand, even a positive $\sigma$-Lebesgue rank one contraction 
on an \text{\rm AM}-space need not to be \text{\rm w}-Lebesgue \cite[Example 3.1)]{AEG2022}.

\medskip
We say that the norm in and ONS $X$ is $(\sigma$-$)${\em order continuous} if $\|x_\alpha\|\to 0$ 
for every net $($sequence$)$ $x_\alpha\convo 0$ in $X$. In certain cases o-convergence is equivalent to ru-convergence.

\begin{lemma}\label{o-conv vs r-conv}{\em
Let $X$ be an OBS with a closed normal cone. 
\begin{enumerate}[$i)$]
\item\
If the norm in $X$ is order continuous then $x_\alpha\convo 0\Longleftrightarrow x_\alpha\convr 0$ in $X$.
\item\
If the norm in $X$ is $\sigma$-order continuous then $x_n\convo 0\Longleftrightarrow x_n\convr 0$ in $X$.
\end{enumerate}}
\end{lemma}

\begin{proof}
First, we show that if $X_+$ is normal then $x_\alpha\convr 0\Longrightarrow x_\alpha\convo 0$ in $X$. 
Suppose $x_\alpha\convr 0$. Then, for some $u\in X_+$ there exists an increasing sequence 
$(\alpha_n)$, $-\frac{1}{n}u\le x_\alpha\le\frac{1}{n}u$ for all $\alpha\ge\alpha_n$. 
Let $0\le y_0\le\frac{1}{n}u$ for all $n\in\mathbb{N}$. Since $\big\|\frac{1}{n}u\big\|\to 0$ then
$y_0=0$ by \cite[Theorem 2.23]{AT2007}. It follows $\frac{1}{n}u\downarrow 0$, and hence $x_\alpha\convo 0$.

\medskip
$i)$\
Let $X$ has order continuous norm and $x_\alpha\convo 0$ in $X$. Take a net $g_\beta\downarrow 0$ in $X$
such that, for each $\beta$ there is  $\alpha_\beta$ satisfying $\pm x_\alpha\le g_\beta$ for $\alpha\ge\alpha_\beta$.
Find an increasing sequence $(\beta_k)$ such that $\sup\limits_{z\in[-g_{\beta_k},g_{\beta_k}]}\|z\|\le k^{-3}$ for all $k\in\mathbb{N}$, and
set $u=\|\cdot\|$-$\sum_{k=1}^\infty kg_{\beta_k}$. Then $\pm x_\alpha\le g_{\beta_k}\le k^{-1}u$ for $\alpha\ge\alpha_{\beta_k}$.
Choose any increasing sequence $(\alpha_k)$ with $\alpha_k\ge\alpha_{\beta_1},\alpha_{\beta_2},\ldots,\alpha_{\beta_k}$.
Then $\pm x_\alpha\le k^{-1}u$ for $\alpha\ge\alpha_k$, and hence $x_\alpha\convr 0$.

\medskip
$ii)$\
The proof is similar.
\end{proof}

\begin{theorem}\label{Lebesgue are onc}
Let $X$ be an OBS with a closed generating normal cone and $Y$ a NS.  
\begin{enumerate}[$i)$]
\item\
If $X$ has $\sigma$-order continuous norm then 
$$
   \text{\rm L}^\sigma_{Leb}(X,Y)=\text{\rm L}^\sigma_{onc}(X,Y) \ \ and \ \ \text{\rm L}^\sigma_{wLeb}(X,Y)=\text{\rm L}^\sigma_{owc}(X,Y).
$$
\item\
If $X$ has order continuous norm then 
$$
   \text{\rm L}_{Leb}(X,Y)=\text{\rm L}_{onc}(X,Y) \ \ and \ \ \text{\rm L}_{wLeb}(X,Y)=\text{\rm L}_{owc}(X,Y).
$$
\end{enumerate}
\end{theorem}

\begin{proof}
We prove only $i)$ as the proof of $ii)$ is similar.
In view of Lemma \ref{directly}, we need to prove $\text{\rm L}^\sigma_{Leb}(X,Y)\subseteq\text{\rm L}^\sigma_{onc}(X,Y)$ 
and $\text{\rm L}^\sigma_{wLeb}(X,Y)\subseteq\text{\rm L}^\sigma_{owc}(X,Y)$. 
Let $x_n\convo 0$ in $X$. Then, $x_n\convr 0$ by Lemma \ref{o-conv vs r-conv}, and hence $\|x_n\|\to 0$ due to normality of $X_+$. 

If $T\in\text{\rm L}^\sigma_{Leb}(X,Y)$ then $T\in\text{\rm L}(X,Y)$ by Theorem \ref{OBS o-cont are obdd}, and hence $\|Tx_n\|\to 0$.
Suppose $T\in\text{\rm L}^\sigma_{wLeb}(X,Y)$. Then again $T\in\text{\rm L}(X,Y)$. So $\|Tx_n\|\to 0$, and hence $Tx_n\convw 0$.
\end{proof}

\begin{corollary}\label{BL Lebesgue are cont}
Let $E$ be a BL with order continuous norm and $Y$ a NS.  Then $\text{\rm L}^\sigma_{Leb}(E,Y)=\text{\rm L}^\sigma_{onc}(E,Y)$,
$\text{\rm L}_{Leb}(E,Y)=\text{\rm L}_{onc}(E,Y)$, $\text{\rm L}^\sigma_{wLeb}(E,Y)=\text{\rm L}^\sigma_{owc}(E,Y)$, and
$\text{\rm L}_{wLeb}(E,Y)=\text{\rm L}_{owc}(E,Y)$.
\end{corollary}

\noindent 
Corollary \ref{BL Lebesgue are cont} is a partial extension of results of \cite[Lemma 1]{JAM2021}, \cite[Lemma 2.1]{AEG2022},  
\cite[Theorem 3]{ZSC2023}, and \cite[Proposition 1.5]{KTA2025} to the OVS setting.

\medskip
\noindent
We give some conditions under which each order bounded operator is order-to-norm continuous.

\begin{proposition}\label{when ob are onc}
Let $X$ be an OBS with a closed normal cone and $Y$ an ONS with a generating normal cone.
\begin{enumerate}[$i)$]
\item\
If $X$ has $\sigma$-order continuous norm then ${\cal L}_{ob}(X,Y)\subseteq\text{\rm L}^\sigma_{onc}(X,Y)$.
\item\
If $X$ has order continuous norm then ${\cal L}_{ob}(X,Y)\subseteq\text{\rm L}_{onc}(X,Y)$.
\end{enumerate}
\end{proposition}

\begin{proof}
We prove only $i)$ as the proof of $ii)$ is similar.
Let $T\in{\cal L}_{ob}(X,Y)$ and $x_n\convo 0$ in $X$. By Lemma \ref{o-conv vs r-conv}, $x_n\convr 0$.
It follows from \cite[Corollary 2.5]{EEG2025} that $Tx_n\convr 0$. The normality of $Y$ provides $\|Tx_n\|\to 0$.
\end{proof}



\smallskip
{}
\end{document}